\documentclass{amsart}
\usepackage{amsmath}
\usepackage{amsthm}
\usepackage{amssymb}

\usepackage[T1]{fontenc}
\usepackage{lmodern}

\usepackage{diagrams}
\input xy
\xyoption{all}

 \newtheorem{thm}{Theorem}[section]
 
 \newtheorem{lem}[thm]{Lemma}
 
 \theoremstyle{definition}
 \newtheorem{defn}[thm]{Definition}
 \theoremstyle{remark}
 \newtheorem{rem}[thm]{Remark}
 \numberwithin{equation}{section}

\newcommand{\ZZ}{\mathbb{Z}}

\newcommand{\PP}{\mathbb{P}}

\newcommand{\Hom}{\mathrm{Hom}}
\newcommand{\HOM}{\mathrm{HOM}}

\newcommand{\Ext}{\mathrm{Ext}}
\newcommand{\cExt}{\mathcal{E}xt}

\newcommand{\uAut}{\underline{\mathrm{Aut}}}

\newcommand{\bPicard}{\mathrm{Picard}}
\newcommand{\uPicard}{\mathcal{P}icard}

\newcommand{\Add}{\mathrm{\mathbf{Add}}}

\newcommand{\coker}{\mathrm{coker}}

\newcommand{\h}{\mathrm{H}}
\newcommand{\R}{\mathrm{R}}

\newcommand{\cD}{\mathcal{D}}

\newcommand{\cK}{\mathcal{K}}

\newcommand{\bS}{\textbf{S}}

\newcommand{\eic}{\mathcal{E}}
\newcommand{\oic}{\mathcal{O}}
\newcommand{\pic}{\mathcal{P}}
\newcommand{\qic}{\mathcal{Q}}

\newcommand{\dic}{\mathcal{D}}

\begin{document}

\title[homology and Picard stack]
{Extensions of Picard stacks\\ and\\ their homological interpretation}

\author{Cristiana Bertolin}

\address{Dip. di Matematica, Universit\`a di Milano, Via C. Saldini 50, I-20133 Milano}
\email{cristiana.bertolin@googlemail.com}






\begin{abstract} 
Let {\bS} be a site. We introduce the notion of extensions of strictly commutative Picard $\bS$-stacks. We define the pull-back, the push-down, and the sum of such extensions and we compute their homological interpretation: if $\pic$ and $\qic$ are two strictly commutative Picard $\bS$-stacks, the equivalence classes of extensions of $\pic$ by $\qic$ are parametrized by the cohomology group $\Ext^1([\pic],[\qic])$, where $[\pic]$ and $[\qic]$ are the complex associated to $\pic$ and $\qic$ respectively.
\end{abstract}


\maketitle


\tableofcontents

\section*{Introduction}

Let $\bS$ be a site. Let $\pic$ and $\qic$ two strictly commutative Picard $\bS$-stacks. We define an 
 extension of $\pic$ by $\qic$ as a strictly commutative Picard $\bS$-stack $\eic$, two additive functors $I:\qic \rightarrow \eic$ and $ J:\eic \rightarrow \pic$, and an isomorphism of additive functors $J \circ I \cong 0$, such that the following equivalent conditions are satisfied:
\begin{itemize}
	\item $\pi_0(J): \pi_0(\eic) \rightarrow \pi_0(\pic)$ is surjective and $I$ induces an equivalence of strictly commutative Picard $\bS$-stacks between $\qic$ and $\ker(J),$ 
	\item $\pi_1(I): \pi_1(\qic) \rightarrow \pi_1(\eic)$ is injective and $J$ induces an equivalence of strictly commutative Picard $\bS$-stacks between $\coker(I)$ and $\pic$.
\end{itemize}
  By \cite{SGA4} \S 1.4 
 there is an equivalence of categories between the category of strictly commutative Picard $\bS$-stacks and the derived category $\cD^{[-1,0]}(\bS)$ of complexes $K$ of abelian sheaves on $\bS$ such that ${\h}^i (K)=0$ for $i \not= -1$ or $0$. 
 Via this equivalence, the above notion of extension of strictly commutative Picard $\bS$-stacks furnishes the notion of extension in the derived category $\cD^{[-1,0]}(\bS)$. Let $K$ and $L$ be two complexes of $\cD^{[-1,0]}(\bS)$.
In this paper we prove that, as for extensions of abelian sheaves on $\bS$, extensions of $K$ by $L$ are parametrized by the cohomology group ${\Ext}^{1}(K,L)$.

More precisely, the extensions of $\pic$ by $\qic$ form a 2-category ${\cExt}(\pic,\qic)$ where
\begin{itemize}
	\item the objects are extensions of $\pic$ by $\qic$,
	\item the 1-arrows are additive functors between extensions,
	\item the 2-arrows are morphisms of additive functors.
\end{itemize}
Equivalence classes of extensions of strictly commutative Picard $\bS$-stacks are endowed with a group law. We denote by ${\cExt}^1(\pic,\qic)$ the group of equivalence classes of objects of ${\cExt}(\pic,\qic)$, by ${\cExt}^{0}(\pic,\qic)$ the group of isomorphism classes of arrows from an object of ${\cExt}(\pic,\qic)$ to itself, and 
by ${\cExt}^{-1}(\pic,\qic) $ the group of automorphisms of an arrow from an object of ${\cExt}(\pic,\qic)$ to itself. With these notation our main Theorem is 

\begin{thm}\label{thm:intro}
Let $\pic$ and $\qic$ be two strictly commutative Picard $\bS$-stacks. Then we have the following isomorphisms of groups 
\begin{description}
	\item[a] ${\cExt}^1(\pic,\qic) \cong {\Ext}^{1}\big([\pic],[\qic]\big)={\Hom}_{\cD(\bS)}\big([\pic],[\qic][1]\big),$
	\item[b] ${\cExt}^0(\pic,\qic) \cong {\Ext}^{0}\big([\pic],[\qic]\big)= {\Hom}_{\cD(\bS)}\big([\pic],[\qic]\big),$
	\item[c] ${\cExt}^{-1}(\pic,\qic) \cong  {\Ext}^{-1}\big([\pic],[\qic]\big)={\Hom}_{\cD(\bS)}\big([\pic],[\qic][-1]\big).$ 
\end{description} 
where $[\pic]$ and $[\qic]$ denote the complex of $\cD^{[-1,0]}(\bS)$ corresponding to $\pic$ and $\qic$ respectively.
\end{thm}

This paper is organized as followed: in Section 1 we recall some basic results on strictly commutative Picard $\bS$-stacks. In Section 2 we introduce the notions of fibered product and fibered sum of strictly commutative Picard $\bS$-stacks. 
In Section 3 we define extensions of strictly commutative Picard $\bS$-stacks and morphisms between such extensions. The results of Section 2 will allow us to define
a group law for equivalence classes of extensions of strictly commutative Picard $\bS$-stacks (Section 4). Finally in Section 5 we prove the main Theorem \ref{thm:intro}.

We finish recalling that the non-abelian analogue of \S3 has been developed by Breen (see in particular \cite{B90}, \cite{B92}), by A. Rousseau in \cite{R}, by Aldrovandi and Noohi in \cite{AN09} and by A. Yekutieli in \cite{Y}.

\section*{Acknowledgment}
The author is grateful to Pierre Deligne for his comments on a first version of this paper.


\section*{Notation}

Let $\bS$ be a site. 
Denote by $\cK(\bS)$ the category of complexes of abelian sheaves on the site $\bS$: all complexes that we consider in this paper are cochain complexes.
Let $\cK^{[-1,0]}(\bS)$ be the subcategory of $\cK(\bS)$ consisting of complexes $K=(K^i)_i$ such that $K^i=0$ for $i \not= -1$ or $0$. The good truncation $ \tau_{\leq n} K$ of a complex $K$ of $\cK(\bS)$ is the following complex: $ (\tau_{\leq n} K)^i= K^i$ for $i <n,  (\tau_{\leq n} K)^n= \ker(d^n)$
and $ (\tau_{\leq n} K)_i= 0$ for $i > n.$ For any $i \in {\ZZ}$, the shift functor $[i]:\cK(\bS) \rightarrow \cK(\bS) $ acts on a  complex $K=(K^n)_n$ as $(K[i])^n=K^{i+n}$ and $d^n_{K[i]}=(-1)^{i} d^{n+i}_{K}$.

Denote by $\cD(\bS)$ the derived category of the category of abelian sheaves on $\bS$, and let $\cD^{[-1,0]}(\bS)$ be the subcategory of $\cD(\bS)$ consisting of complexes $K$ such that ${\h}^i (K)=0$ for $i \not= -1$ or $0$. If $K$ and $K'$ are complexes of $\cD(\bS)$, the group ${\Ext}^i(K,K')$ is by definition ${\Hom}_{\cD(\bS)}(K,K'[i])$ for any $i \in {\ZZ}$. Let ${\R}{\Hom}(-,-)$ be the derived functor of the bifunctor ${\Hom}(-,-)$. The cohomology groups\\ ${\h}^i\big({\R}{\Hom}(K,K') \big)$ of 
${\R}{\Hom}(K,K')$ are isomorphic to ${\Hom}_{\cD(\bS)}(K,K'[i])$.

A \emph{2-category} $\mathcal{A}=(A,C(a,b),K_{a,b,c},U_{a})_{a,b,c \in A}$ is given by the following data:
\begin{itemize}
  \item a set $A$ of objects $a,b,c, ...$;
  \item for each ordered pair $(a,b)$ of objects of $A$, a category $C(a,b)$;
  \item for each ordered triple $(a,b,c)$ of objects $A$, a functor
      $K_{a,b,c}:C(b,c) \times C(a,b) \longrightarrow C(a,c),$
      called composition functor. This composition functor have to satisfy the associativity law;
  \item for each object $a$, a functor $U_a:\mathbf{1} \rightarrow C(a,a)$ where \textbf{1} is the terminal category (i.e. the category with one object, one arrow), called unit functor. This unit functor have to provide a left and right identity for the composition functor.
\end{itemize}

This set of axioms for a 2-category is exactly like the set of axioms for a category in which the arrows-sets ${\Hom}(a,b)$ have been replaced by the categories $C(a,b)$.
We call the categories $C(a,b)$ (with $a,b \in A$) the \emph{categories of morphisms} of the 2-category $\mathcal{A}$: the objects of $C(a,b)$ are the \emph{1-arrows} of $\mathcal{A}$ and the arrows of $C(a,b)$ are the \emph{2-arrows} of $\mathcal{A}$.

Let $\mathcal{A}=(A,C(a,b),K_{a,b,c},U_{a})_{a,b,c \in A} $ and $ \mathcal{A}'=(A',C(a',b'),K_{a',b',c'},U_{a'})_{a',b',c' \in A'}$ be two 2-categories. A \emph{2-functor} (called also a \emph{morphism of 2-categories})
$$(F,F_{a,b})_{a,b \in A}: \mathcal{A} \longrightarrow \mathcal{A}'$$
consists of
\begin{itemize}
  \item an application $F: A \rightarrow A'$ between the objects of $\mathcal{A}$ and the objects of $\mathcal{A}'$,
  \item a family of functors $F_{a,b}:C(a,b) \rightarrow C(F(a),F(b))$ (with $a,b \in A$) which are compatible with the composition functors and with the unit functors of $\mathcal{A}$ and $\mathcal{A}'$.
\end{itemize}

\section{Recall on strictly commutative Picard stacks}

Let {\bS} be a site. For the notions of {\bS}-pre-stack, {\bS}-stack and morphisms of {\bS}-stacks we refer to \cite{G} Chapter II 1.2.

A \emph{strictly commutative Picard {\bS}-stack} is an {\bS}-stack of groupoids $\pic$ endowed with a functor $ +: \pic \times_{\bS} \pic \rightarrow \pic, ~~(a,b) \mapsto a+b$, and two 
  natural isomorphisms of associativity $\sigma$ and of commutativity $\tau$, which are described by the functorial isomorphisms
\begin{eqnarray}
\label{ass}  \sigma_{a,b,c} &:& (a + b) + c \stackrel{\cong}{\longrightarrow} a + ( b + c) \qquad \forall~ a,b,c \in \pic, \\
\label{com} \tau_{a,b} &:& a + b \stackrel{\cong}{\longrightarrow} b + a \qquad \forall~ a,b \in \pic;
\end{eqnarray}
such that for any object $U$ of $\bS$, $(\pic (U),+,\sigma, \tau)$ is a strictly commutative Picard category (i.e. it is possible to make the sum of two objects of $\pic (U)$ and this sum is associative and commutative, see \cite{SGA4} 1.4.2 for more details). Here "strictly" means that $\tau_{a,a}$ is the identity 
for all $a \in \pic$.
Any strictly commutative Picard {\bS}-stack admits a global neutral object $e$ and the sheaf of automorphisms of the neutral object ${\uAut}(e)$ is abelian.

Let $\pic$ and $\qic$ be two strictly commutative Picard {\bS}-stacks.
 An \emph{additive functor} $(F,\sum):\pic \rightarrow \qic $
between strictly commutative Picard {\bS}-stacks is a morphism of {\bS}-stacks (i.e. a cartesian {\bS}-functor, see \cite{G} Chapter I 1.1) endowed with a natural isomorphism $\sum$ which is described by the functorial isomorphisms
\[\sum_{a,b}:F(a+b) \stackrel{\cong}{\longrightarrow} F(a)+F(b)
\qquad \forall~ a,b \in \pic \]
and which is compatible with the natural isomorphisms $\sigma$ and $\tau$ of
$\pic$ and $\qic$. A \emph{morphism of additive functors $u:(F,\sum) \rightarrow (F',\sum') $} is an {\bS}-morphism of cartesian {\bS}-functors (see \cite{G} Chapter I 1.1) which is compatible with the natural isomorphisms $\sum$ and $\sum'$ of $F$ and $F'$ respectively.
We denote by ${\Add}_{\bS} (\pic,\qic)$
the category whose objects are additive functors from $\pic$ to $\qic$ and whose arrows are morphisms of additive functors. The category ${\Add}_{\bS} (\pic,\qic)$ is a groupoid, i.e. any morphism of additive functors is an isomorphism of additive functors.

An \emph{equivalence of strictly commutative Picard {\bS}-stacks} between $\pic$ and $\qic$ is an additive functor $(F,\sum):\pic \rightarrow \qic$ with $F$ an equivalence of {\bS}-stacks. Two strictly commutative Picard {\bS}-stacks are \emph{equivalent as strictly commutative Picard {\bS}-stacks} if there exists an equivalence of strictly commutative Picard {\bS}-stacks between them.

To any strictly commutative Picard {\bS}-stack $\pic$, we associate the sheaffification $\pi_0(\pic)$ of the pre-sheaf which associates to each object $U$ of $\bS$ the group of isomorphism classes of objects of $\pic(U)$,
 the sheaf $\pi_1(\pic)$ of automorphisms ${\uAut}(e)$ of the neutral object of $\pic$, and an element $\varepsilon(\pic)$ of ${\Ext}^2(\pi_0(\pic),\pi_1(\pic))$. 
Two strictly commutative Picard {\bS}-stacks $\pic$ and $\pic'$ are equivalent as strictly commutative Picard {\bS}-stacks if and only if $\pi_i(\pic)$ is isomorphic to $\pi_i(\pic')$ for $i=0,1$ and $\varepsilon(\pic)=\varepsilon(\pic')$ (see Remark \ref{012}).

 A \emph{strictly commutative Picard {\bS}-pre-stack} is an {\bS}-pre-stack of groupoids $\pic$ endowed with a functor $ +: \pic \times_{\bS} \pic \rightarrow \pic$ and
two natural isomorphisms of associativity $\sigma$ (\ref{ass}) and of commutativity $\tau$ (\ref{com}),
such that for any object $U$ of $\bS,$ $(\pic (U),+,\sigma, \tau)$ is a strictly commutative Picard category. If $\pic$ is a strictly commutative Picard \bS-pre-stack, there exists modulo a unique equivalence
one and only one pair $(a \pic, j)$ where $a\pic$ is a strictly commutative Picard \bS-stack and $j:  \pic \rightarrow a \pic$ is an additive functor. $(a \pic, j)$ is \emph{the strictly commutative Picard \bS-stack generated by $\pic$.}

To each complex
$K=[K^{-1} \stackrel{d}{\rightarrow}K^0]$
of $\cK^{[-1,0]}(\bS)$, we associate a strictly commutative Picard {\bS}-stack $st(K)$ which is the  {\bS}-stack generated by the following strictly commutative Picard {\bS}-pre-stack $pst(K)$: 
 for any object $U$ of {\bS}, the objects of $pst(K)(U)$ are the elements of $K^0(U)$, and
  if $x$ and $y$ are two objects of $pst(K)(U)$ (i.e. $x,y \in K^0(U)$), an arrow of $pst(K)(U)$ from $x$ to $y$ is an element $f$ of $K^{-1}(U)$ such that $df=y-x$.
A morphism of complexes $g: K \rightarrow L$ induces an additive functor 
$st(g):st(K) \rightarrow st(L)$
between the strictly commutative Picard {\bS}-stacks associated to the complexes $K$ and $L$. 

In~\cite{SGA4} \S 1.4 Deligne proves the following links between strictly commutative Picard {\bS}-stacks and complexes of $\cK^{[-1,0]}(\bS)$, between additive functors and morphisms of complexes and between morphisms of additive functors and homotopies of complexes:
\begin{itemize}
	\item for any strictly commutative Picard $\bS$-stack $\pic$ there exists a complex $K$ of $\cK^{[-1,0]}(\bS)$ such that $\pic = st(K)$;
	\item if $K,L$ are two complexes of $\cK^{[-1,0]}(\bS)$, then for any additive functor $F: st(K) \rightarrow st(L)$ there exists a quasi-isomorphism $k:K' \rightarrow K$ and a morphism of complexes $l:K' \rightarrow L$ such that $F$ is isomorphic as additive functor to $st(l) \circ st(k)^{-1}$;
	\item if $f,g: K \rightarrow L$ are two morphisms of complexes of $\cK^{[-1,0]}(\bS)$, then
\begin{equation}
{\Hom}_{{\Add}_{\bS}(st(K),st(L))}(st(f),st(g)) \cong \Big\{ \mathrm{homotopies}~ H:K \rightarrow L ~~|~~ g-f=dH+Hd \Big\}.
\label{eq:homotopies}
\end{equation}
\end{itemize}
Denote by ${\bPicard}(\bS)$ the category whose objects are small strictly commutative Picard $\bS$-stacks and whose arrows are isomorphism classes of additive functors. We can summarize the above dictionary between strictly commutative Picard $\bS$-stacks and complexes of abelian sheaves on $\bS$ with the following Theorem:

\begin{thm}
The functor
\begin{eqnarray}
\label{st} st: \cD^{[-1,0]}(\bS) &\longrightarrow & {\bPicard}(\bS) \\
 \nonumber  K & \mapsto & st(K) \\
 \nonumber K \stackrel{f}{\rightarrow} L & \mapsto &  st(K) \stackrel{st(f)}{\rightarrow} st(L)
\end{eqnarray}
is an equivalence of categories.
\end{thm}

We denote by $[\,\,]$ the inverse equivalence of $st$. \\
Let ${\uPicard}(\bS)$ be the 2-category of strictly commutative Picard $\bS$-stacks whose objects are strictly commutative Picard $\bS$-stacks and whose categories of morphisms are the categories ${\Add}_{\bS}(\pic,\qic)$ (i.e.
       the 1-arrows are additive functors between strictly commutative Picard $\bS$-stacks and the 2-arrows are morphisms of additive functors).

\begin{thm}
Via the functor $st$, there exists a 2-functor between
\begin{description}
  \item[(a)] the 2-category whose objects and 1-arrows are the objects and the arrows of the category $\cK^{[-1,0]}(\bS)$ and whose 2-arrows are the homotopies between 1-arrows (i.e. $H$ such that $g-f=dH+Hd$ with $f,g: K \rightarrow L$ 1-arrows),
  \item[(b)] the 2-category ${\uPicard}(\bS)$.
\end{description}
\end{thm}

\begin{rem}\label{012}
Let $K= [K^{-1} \stackrel{d}{\rightarrow}K^0]$ be a complex of $\cD^{[-1,0]}(\bS)$.
The long exact sequence 
\[
0 \longrightarrow {\h}^{-1}(K)  \longrightarrow K^{-1}  \stackrel{d}{\longrightarrow} K^0 \longrightarrow {\h}^{0}(K) \longrightarrow 0
\]
is an element of ${\Ext}^2({\h}^{0}(K),{\h}^{-1}(K))$ that we denote by $\varepsilon(K)$.
The sheaves ${\h}^{0},{\h}^{-1}$ and the element $\varepsilon$ of ${\Ext}^2({\h}^{0},{\h}^{-1})$ classify objects of $\cD^{[-1,0]}(\bS)$ modulo isomorphisms. 
Through the equivalence of categories (\ref{st}), the above classification of objects of $\cD^{[-1,0]}(\bS)$ is equivalent to the classification of 
strictly commutative Picard $\bS$-stacks via
 the sheaves $\pi_0, \pi_1$ and the invariant $ \varepsilon \in {\Ext}^2(\pi_0,\pi_1)$. In particular $ \pi_0(\pic) = {\h}^{0}([\pic]) ,  \pi_1(\pic) = {\h}^{-1}([\pic]),
 \varepsilon(\pic) = \varepsilon([\pic]) .$
\end{rem}

\par \textbf{Example} Let $\pic$ and $\qic$ be two strictly commutative Picard {\bS}-stacks. Let 
$${\HOM}(\pic,\qic)$$
be the following strictly commutative Picard {\bS}-stack:
\begin{itemize}
  \item for any object $U$ of $\bS$, the objects of the category ${\HOM}(\pic,\qic)(U)$ are additive functors from ${\pic}_{|U}$ to ${\qic}_{|U}$ and its arrows are morphisms of additive functors;
  \item the functor $+: {\HOM}(\pic,\qic) \times {\HOM}(\pic,\qic) \rightarrow {\HOM}(\pic,\qic)$ is defined by the formula
      \[(F_1+F_2)(a)=F_1(a)+F_2(a) \qquad \forall~ a \in \pic\]
       and the natural isomorphism
       \[ \sum:(F_1+F_2)(a+b) \stackrel{\cong}{\longrightarrow} (F_1+F_2)(a)+(F_1+F_2)(b)\]
       is given by the commutative diagram
  \[    \xymatrix{
     (F_1+F_2)(a+b)\ar[r]^{\sum}  \ar@{=}[d]& (F_1+F_2)(a)+(F_1+F_2)(b) \ar@{=}[r]& F_1(a)+F_2(a)+F_1(b)+F_2(b)   \\
   F_1(a+b)+F_2(a+b) \ar[rr]_{\sum_{F_1} + \sum_{F_2}}  &  & F_1(a)+F_1(b)+F_2(a)+F_2(b) \ar[u]^{Id + \tau_{F_1(b),F_2(a)}+Id} .
}\]
  \item the natural isomorphisms of associativity $\sigma$ and of commutativity $\tau$ of ${\HOM}(\pic,\qic)$ are defined via the analogous natural isomorphisms of $\qic$.
\end{itemize}
Because of equality (\ref{eq:homotopies}) and of Theorem \ref{st}, we have the following equality in the derived category 
$\cD^{[-1,0]}(\bS)$

\begin{lem}\label{lem:HOM} 
$$ [{\HOM}(\pic,\qic)] = \tau_{\leq 0}{\R}{\Hom}\big([\pic],[\qic]\big). $$
\end{lem}

We can define the following bifunctor on ${\bPicard}(\bS) \times {\bPicard}(\bS)$
\begin{eqnarray}
\nonumber {\HOM}: {\bPicard}(\bS) \times {\bPicard}(\bS)&\longrightarrow & {\bPicard}(\bS) \\
 \nonumber  (\pic,\qic) & \mapsto & {\HOM}(\pic,\qic).
\end{eqnarray}
\section{Operations on strictly commutative Picard stacks}

We start defining the product of two strictly commutative Picard $\bS$-stacks.
Let $\pic$ and $\qic$ be two strictly commutative Picard $\bS$-stacks. 

\begin{defn}
The \emph{product of $\pic$ and $\qic$} is the strictly commutative Picard $\bS$-stack 
$\pic \times \qic$ defined as followed:
\begin{itemize}
  \item for any object $U$ of $\bS$, an object of the category $\pic \times \qic(U)$ is a pair $(p,q)$ of objects with $p$ an object of $\pic(U)$ and $q$ an object of $\qic(U)$;
  \item for any object $U$ of $\bS$, if $(p,q)$ and $(p',q')$ are two objects of 
  $\pic \times \qic(U)$, an arrow of $\pic \times \qic(U)$ from $(p,q)$ to $(p',q')$ is a pair $(f,g)$ of arrows with $f:p \rightarrow  p' $ an arrow of $\pic(U)$ and $g:q \rightarrow  q' $ an arrow of $\qic(U)$;
  \item the functor $+: (\pic \times \qic) \times (\pic \times \qic) \rightarrow \pic \times \qic$ is defined by the formula
      \[(p,q)+(p',q')=(p+p',q+q') \]
for any $p,p' \in \pic$ and $q,q' \in \qic$;
  \item the natural isomorphisms of associativity $\sigma$ and of commutativity $\tau$ of $\pic \times \qic$ are defined via the analogous natural isomorphisms of $\pic$ and  $\qic$.
\end{itemize}
\end{defn}

In the derived category $\cD^{[-1,0]}(\bS)$ we have the following equality
$$[\pic \times \qic ]=[\pic ] +[\qic ]$$
which implies the following equality of abelian sheaves 
$$\pi_i(\pic \times \qic ) = \pi_i(\pic) + \pi_i(\qic) \qquad \mathrm{for} \qquad i=0,1.$$

Now we define the fibered product (called also the pull-back) and the fibered sum (called also the push-down) of strictly commutative Picard $\bS$-stacks. Let $G:\pic \rightarrow \qic$ and $F:\pic' \rightarrow \qic$ be additive functors between strictly commutative Picard $\bS$-stacks.

\begin{defn}\label{def:pull-back}
 The \emph{fibered product of $\pic$ and $\pic'$ over $\qic$ via $F$ and $G$} is the strictly commutative Picard $\bS$-stack $\pic \times_\qic \pic'$ defined as followed: 
\begin{itemize}
  \item for any object $U$ of {\bS}, the objects of the category $(\pic \times_\qic \pic')(U)$ are triplets $(p,p',f)$ where $p$ is an object of $\pic(U)$, $p' $ is an object of $\pic'(U)$ and $f:G(p) \stackrel{\cong}{\rightarrow}F(p')$ is an isomorphism of $\qic(U)$ between $G(p)$ and $F(p')$;
  \item for any object $U$ of {\bS}, if $(p_1,p'_1,f)$ and $(p_2,p'_2,g)$ are two objects of
  $(\pic \times_\qic \pic')(U)$, an arrow of $(\pic \times_\qic \pic')(U)$ 
 from  $(p_1,p'_1,f)$ to $(p_2,p'_2,g)$ is a pair $(f,g)$ of arrows with
  $\alpha:p_1 \rightarrow p_2$ of arrow of $\pic(U)$ and $\beta:p'_1 \rightarrow p'_2$ an arrow of $\pic'(U)$ such that $ g \circ G(\alpha) = F(\beta) \circ f$;
  \item the functor $+: (\pic \times_\qic \pic') \times (\pic \times_\qic \pic') \rightarrow \pic \times_\qic \pic'$ is induced by the functors $+: \pic \times \pic \rightarrow \pic$ and $+: \pic' \times \pic' \rightarrow \pic'$;
  \item the natural isomorphisms of associativity $\sigma$ and of commutativity $\tau$ are induced by the analogous natural isomorphisms of $\pic$ and $\pic'$.
\end{itemize}
\end{defn}

The fibered product $\pic \times_\qic \pic'$ is also called the \emph{pull-back $F^*\pic$ of $\pic$ via $F:\pic' \rightarrow \qic$} or the \emph{pull-back $G^*\pic'$ of $\pic'$ via $G:\pic \rightarrow \qic$}. It is endowed with two additive functors $Pr_1: \pic \times_\qic \pic' \rightarrow \pic$ and $ Pr_2: \pic \times_\qic \pic' \rightarrow \pic'.$ Moreover it satisfies the following universal property: given two additive functors $H: \oic \rightarrow \pic$ and 
$K: \oic \rightarrow \pic'$, and given an isomorphism of additive functors $ F \circ K \cong G \circ H$, then there exists a unique additive functor $U: \oic \rightarrow \pic \times_\qic \pic'$ such that we have two isomorphisms of additive functors $H \cong Pr_1 \circ U$ and $K \cong Pr_2 \circ U$.
The following square formed by the fibered sum $\pic \times_\qic \pic'$
\[    \xymatrix{
 \pic \times_\qic \pic'  \ar[r]^{Pr_2}  \ar[d]_{Pr_1}& \pic' \ar[d]^{ F} \\
  \pic \ar[r]_{G}  & \qic
}\]
is called a \emph{pull-back square} or a \emph{cartesian square}.

If $[\pic]=[K^{-1}  \stackrel{d_K}{\rightarrow} K^0],[\pic']=[L^{-1}  \stackrel{d_L}{\rightarrow} L^0]$ and $[\qic]=[M^{-1}  \stackrel{d_M}{\rightarrow} M^0],$ in the derived category $\cD^{[-1,0]}(\bS)$ we have the following equality
$$[\pic \times_\qic \pic' ]=[K^{-1} \times_{M^{-1}} L^{-1} \stackrel{d_K \times_{d_M}d_L}{\longrightarrow} K^{0} \times_{M^{0}} K^{0}]$$
where for $i=-1,0$ the abelian sheaf $K^{i} \times_{M^{i}} L^{i}$ is the fibered product of $K^i$ and of $L^i$ over $M^i$ and the morphism of abelian sheaves $d_K \times_{d_M}d_L $ is given by the universal property of the fibered product 
$K^{0} \times_{M^{0}} K^{0}$. 

Remark that we have the exact sequences of abelian sheaves 
$$0 \longrightarrow \pi_1(\pic \times_\qic \pic') \longrightarrow \pi_1(\pic)+ \pi_1(\pic').$$

Now we introduce the dual notion of fibered product: the fibered sum. Let $G:\qic \rightarrow \pic$ and $F:\qic \rightarrow \pic'$ be additive functors between strictly commutative Picard $\bS$-stacks.

\begin{defn}\label{def:push-down}
 The \emph{fibered sum of $\pic$ and $\pic'$ under $\qic$ via $F$ and $G$} is the strictly commutative Picard $\bS$-stack $\pic +^\qic \pic'$ generated by the following strictly commutative Picard $\bS$-pre-stack $\dic$:
\begin{itemize}
  \item for any object $U$ of {\bS}, the objects of the category $\dic(U)$ are the objects of the category $(\pic \times \pic')(U)$, i.e. pairs $(p,p')$ with $p$ an object of $\pic(U)$ and  $p'$ an object of $\pic'(U)$;
  \item for any object $U$ of {\bS}, if $(p_1,p'_1)$ and $(p_2,p'_2)$ are two objects of $\dic(U)$, an arrow of $\dic (U)$ from $(p_1,p'_1)$ to $(p_2,p'_2)$ is an isomorphism class of triplets  
  $(q,\alpha,\beta)$ with $q$ an object of $\qic (U)$, $\alpha: p_1 + G(q) \rightarrow p_2$ an arrow of $\pic(U)$ and $\beta: p'_1 + F(q) \rightarrow p'_2$ an arrow of $\pic'(U)$;
  \item the functor $+: \dic \times \dic \rightarrow \dic$ is induced by the functors $+: \pic \times \pic \rightarrow \pic$ and $+: \pic' \times \pic' \rightarrow \pic'$;
  \item the natural isomorphisms of associativity $\sigma$ and of commutativity $\tau$ are induced by the analogous natural isomorphisms of $\pic$ and $\pic'$.
\end{itemize}
\end{defn}

The fibered sum $\pic +^\qic \pic'$ is also called the \emph{push-down $F_*\pic$ of $\pic$ via $F:\qic \rightarrow \pic'$} or the \emph{push-down $G_*\pic'$ of $\pic'$ via $G:\qic \rightarrow \pic$}. It is endowed with two additive functors $In_1: \pic \rightarrow \pic +^\qic \pic'$ and $In_2: \pic' \rightarrow \pic +^\qic \pic'$. Moreover it satisfies the following universal property: given two additive functors $H: \pic \rightarrow \oic$ and 
$K: \pic' \rightarrow \oic$, and given an isomorphism of additive functors $ K \circ F \cong H \circ G$, then there exists a unique additive functor $U: \pic +^\qic \pic' \rightarrow \oic $ such that we have two isomorphisms of additive functors $H \cong U \circ In_1 $ and $K \cong U \circ In_2$.
The following square formed by the fibered sum $\pic +^\qic \pic'$ 
\[    \xymatrix{
 \qic \ar[r]^F  \ar[d]_G& \pic' \ar[d]^{In_2} \\
  \pic \ar[r]_{In_1}  & \pic +^\qic \pic' 
}\]
is called a \emph{push-down square} or a \emph{cocartesian square}.

If $[\pic]=[K^{-1}  \stackrel{d_K}{\rightarrow} K^0],[\pic']=[L^{-1}  \stackrel{d_L}{\rightarrow} L^0]$ and $[\qic]=[M^{-1}  \stackrel{d_M}{\rightarrow} M^0],$ in the derived category $\cD^{[-1,0]}(\bS)$ we have the following equality
$$[\pic +^\qic \pic' ]=[K^{-1} +^{M^{-1}} L^{-1} \stackrel{d_K +^{d_M}d_L}{\longrightarrow} K^{0} +^{M^{0}} K^{0}]$$
where for $i=-1,0$ the abelian sheaf $K^{i} +^{M^{i}} L^{i}$ is the fibered sum of $K^i$ and of $L^i$ under $M^i$ and the morphism of abelian sheaves $d_K +^{d_M}d_L $ is given by the universal property of the fibered product 
$K^{-1} +^{M^{-1}} K^{-1}$.

Remark that we have the exact sequences of abelian sheaves 
$$\pi_0(\pic)+ \pi_0 (\pic') \longrightarrow \pi_0(\pic +^\qic \pic') \longrightarrow 0.$$

\section{Extensions of strictly commutative Picard stacks}

Let $\pic$ and $\qic$ be two strictly commutative Picard $\bS$-stacks. Consider an additive functor $F:\pic \rightarrow \qic$. Denote by $\mathbf{1}$
the strictly commutative Picard $\bS$-stack such that for any object $U$ of $\bS$,
$\mathbf{1}(U)$ is the category with one object and one arrow. 

\begin{defn}\label{def:kercoker}
The \emph{kernel of $F$, $\ker(F),$} is the fibered product $ \pic \times_\qic \mathbf{1}$ of $\pic$ and $ \mathbf{1}$ over $\qic$ via $F:\pic \rightarrow \qic$ and the additive functor $\mathbf{1}: \mathbf{1} \rightarrow \qic$. \\
The \emph{cokernel of $F$, $\coker(F),$} is the fibered sum $\mathbf{1} +^\pic \qic$ of $ \mathbf{1}$ and $\qic$ under $\pic$ via $F: \pic \rightarrow \qic$ and the additive functor $\mathbf{1}: \pic \rightarrow \mathbf{1} $. 
\end{defn}

We have the cartesian and cocartesian squares 
\[  \xymatrix{
 \ker(F) \ar[r]  \ar[d] & \mathbf{1} \ar[d]^{\mathbf{1}} \\
 \pic \ar[r]_F  & \qic
}
\qquad    \qquad
    \xymatrix{
 \pic \ar[r]^F  \ar[d]_{\mathbf{1}}& \qic \ar[d] \\
  \mathbf{1} \ar[r]  & \coker(F)
}\]
 and the exact sequences of abelian sheaves 
\begin{equation}
\label{eq:ker-coker} 0 \longrightarrow \pi_1(\ker(F)) \longrightarrow \pi_1(\pic) \qquad   \qquad 
\pi_0(\qic) \longrightarrow \pi_0(\coker(F)) \longrightarrow 0.	
\end{equation}

Explicitly, according to Definition \ref{def:pull-back} the kernel of $F$ is the strictly commutative Picard $\bS$-stack $\ker(F)$ where
\begin{itemize}
  \item for any object $U$ of {\bS}, the objects of the category $\ker(F)(U)$ are pairs $(p,f)$ where $p$ is an object of $\pic(U)$ and $f:F(p) \stackrel{\cong}{\rightarrow} e$ is an isomorphism between $F(p)$ and the neutral object $e$ of $\qic$;
  \item for any object $U$ of {\bS}, if $(p,f)$ and $(p',f')$ are two objects of
  $\ker(F)(U)$, an arrow $\alpha:(p,f) \rightarrow (p',f')$ of $\ker(F)(U)$ is an arrow $\alpha:p \rightarrow p'$ of $\pic(U)$ such that $f' \circ F(\alpha) =f. $
\end{itemize}
By definition \ref{def:push-down} the cokernel of $F$ is the strictly commutative Picard $\bS$-stack $\coker(F)$ generated by the following strictly commutative Picard $\bS$-pre-stack\\
 $\coker'(F)$ where
\begin{itemize}
  \item for any object $U$ of {\bS}, the objects of $\coker'(F)(U)$ are the objects of $\qic(U)$;
   \item for any object $U$ of {\bS}, if $q$ and $q'$ are two objects of
   $\coker'(F)(U)$ (i.e. objects of $\qic(U)$), an arrow of $\coker'(F)(U)$ from $q$ to $q'$ is an isomorphism class of pairs $(p, \alpha)$ with
   $p$ an object of $\pic(U)$ and $\alpha: q + F(p) \rightarrow q'$ an arrow of $\qic(U)$.
\end{itemize}

\begin{defn}\label{def:extpic}
An \emph{extension $\eic=(\eic,I,J) $ of $\pic$ by $\qic$} 
\begin{equation}
\qic \stackrel{I}{\longrightarrow} \eic \stackrel{J}{\longrightarrow} \pic  
\end{equation}
consists of 
\begin{enumerate}
	\item a strictly commutative Picard $\bS$-stack $\eic$,
	\item two additive functors $I:\qic \rightarrow \eic$ and $ J:\eic \rightarrow \pic$,
	\item an isomorphism of additive functors between the composite $J \circ I$ and the trivial additive functor: $J \circ I \cong 0,$
\end{enumerate}
such that the following equivalent conditions are satisfied:
   \begin{description}
     \item[(a)] $\pi_0(J): \pi_0(\eic) \rightarrow \pi_0(\pic)$ is surjective and $I$ induces an equivalence of strictly commutative Picard $\bS$-stacks between $\qic$ and $\ker(J);$
     \item[(b)] $\pi_1(I): \pi_1(\qic) \rightarrow \pi_1(\eic)$ is injective and $J$ induces an equivalence of strictly commutative Picard $\bS$-stacks between $\coker(I)$ and $\pic$. 
   \end{description}
\end{defn}

The additive functors $I: \qic \rightarrow \eic$ and $J: \eic \rightarrow \pic$ induce the sequences of abelian sheaves
\[
0 \longrightarrow \pi_1(\qic) \stackrel{\pi_1(I)}{\longrightarrow} \pi_1(\eic) \stackrel{\pi_1(J)}{\longrightarrow} \pi_1(\pic)\]
\[
\pi_0(\qic) \stackrel{\pi_0(I)}{\longrightarrow} \pi_0(\eic) \stackrel{\pi_0(J)}{\longrightarrow} \pi_0(\pic)\longrightarrow 0\]
which are exact in $\pi_1(\qic)$ and $ \pi_0(\pic)$ because of the equivalences of strictly commutative Picard $\bS$-stacks $\qic \cong \ker(J)$ and $\coker(I) \cong \pic.$ According to \cite{AN09} Proposition 6.2.6, we can say 
 more about these two sequences: in fact there exists a connecting morphism of abelian sheaves 
\[ \delta : \pi_1(\pic) \longrightarrow \pi_0(\qic) \]
leading to the long exact sequence
\begin{equation}
0 \longrightarrow \pi_1(\qic) \stackrel{\pi_1(I)}{\longrightarrow} \pi_1(\eic) \stackrel{\pi_1(J)}{\longrightarrow} \pi_1(\pic) \stackrel{\delta}{\longrightarrow} \pi_0(\qic) \stackrel{\pi_0(I)}{\longrightarrow} \pi_0(\eic) \stackrel{\pi_0(J)}{\longrightarrow} \pi_0(\pic) \longrightarrow 0.
\label{prop:longexseq}
\end{equation}
Explicitly the connecting morphism $\delta: \pi_1(\pic) \rightarrow \pi_0(\qic) $
is defined as followed: if $f:e_\pic \rightarrow e_\pic$ is an element of $\pi_1(\pic)(U)$ with $U$ an object of $\bS$, then $\delta (f)$ represent the isomorphism class of the element 
\[(e_{\eic},f \circ 1_J)\]
 of $\ker(J)(U) \cong \qic(U)$, where $1_J:J(e_\eic) \stackrel{\cong}{\rightarrow} e_{\pic}$ is the isomorphism resulting from the additivity of the functor $J:\eic \rightarrow \pic$ (here $e_\eic$ and $e_\pic$ are the neutral objects of $\eic$ and $\pic$ respectively).\\

Let $\pic, \qic, \pic'$ and $\qic'$ be strictly commutative Picard $\bS$-stacks. Let $\eic=(\eic,I,J)$ be an extension of $\pic$ by $\qic$ and let $\eic'=(\eic',I',J')$ be an extension of $\pic'$ by $\qic'$.

\begin{defn}\label{def:morfismi}
 A \emph{morphism of extensions}
\[(F,G,H): \eic \longrightarrow \eic' \]
consists of 
\begin{enumerate}
	\item three additive functors $F:\eic \rightarrow \eic', G: \pic \rightarrow \pic', H:\qic \rightarrow \qic'$,
	\item two  isomorphisms of additive functors $J' \circ F \cong G \circ J$ and $ F \circ I \cong I'\circ H$,
\end{enumerate}
which are compatible with the isomorphisms of additive functors  $J \circ I \cong 0$ and $J' \circ I' \cong 0$ underlying the extensions $\eic$ and $\eic'$, i.e. the composite 
\[ 0 \stackrel{\cong}{\longleftrightarrow} G \circ 0 \stackrel{\cong}{\longleftrightarrow} G \circ J \circ I  \stackrel{\cong}{\longleftrightarrow}  J' \circ F \circ I \stackrel{\cong}{\longleftrightarrow} J' \circ I'\circ H  \stackrel{\cong}{\longleftrightarrow}  0 \circ H \stackrel{\cong}{\longleftrightarrow} 0\]
should be the identity.
\end{defn}

The three additive functors $F, G$ and $ H$ furnish the following commutative diagram modulo isomorphisms of additive functors 
\[\xymatrix{
   \qic  \ar[d]_{H} \ar[r]^{I} & \eic \ar[d]_{F} \ar[r]^{J} & \pic  \ar[d]^{G}  \\
  \qic' \ar[r]_{I'} & \eic' \ar[r]_{J'} & \pic' .
}
\]

Fix two strictly commutative Picard $\bS$-stacks $\pic$ and $\qic$.
The extensions of $\pic$ by $\qic$ form a 2-category 
$${\cExt}(\pic,\qic)$$
 where
\begin{itemize}
	\item the objects are extensions of $\pic$ by $\qic$,
	\item the 1-arrows are morphisms of extensions, i.e. additive functors between extensions,
	\item the 2-arrows are morphisms of additive functors.
\end{itemize}

Now we show which objects of the derived category $\cD^{[-1,0]}(\bS)$ correspond via the equivalence of categories~(\ref{st}) to the strictly commutative Picard $\bS$-stacks $\ker(F)$, $\coker(F)$ and $\pic=(\pic,I,J)$.

\begin{lem}\label{lem:kercoker}
Let $K=[K^{-1} \stackrel{d^K}{\rightarrow}K^0]$ and $L=[L^{-1} \stackrel{d^L}{\rightarrow}L^0]$ be complexes of $\cK^{[-1,0]}(\bS)$. 
Let $F: st(K)  \rightarrow st(L)$ be an additive functor induced by a morphism of complexes $f=(f^{-1},f^0): K \rightarrow L$.
The strictly commutative Picard $\bS$-stacks $\ker(F)$ and $\coker(F)$
correspond via the equivalence of categories~(\ref{st}) to the following complexes of $\cK^{[-1,0]}(\bS):$
\begin{eqnarray}
\label{eq:[ker]} [\ker(F)] & = & \tau_{ \leq 0}\big( MC(f)[-1]\big) = \big[K^{-1} ~ \stackrel{(f^{-1},-d^K)}{\longrightarrow} ~  \ker(d^L,f^0)\big]\\
\label{eq:[coker]} [\coker(F)] & = & \tau_{\geq -1}MC(f)  = \big[\coker(f^{-1},-d^K) ~ \stackrel{(d^L,f^0)}{\longrightarrow} ~ L^0 \big]
\end{eqnarray}
where $\tau$ denotes the good truncation and $MC(f)$ is the mapping cone of the morphism $f$.
\end{lem}

\begin{proof} It is enough to show that the strictly commutative $\bS$-pre-stacks associated to $\coker(F)$ is equivalent to the one associated to $\tau_{\geq -1}MC(f)$, since for each strictly commutative $\bS$-pre-stack $\pic$, the strictly commutative $\bS$-stack generated by $\pic$ is unique modulo a unique equivalence (idem for $\ker(F)$). Explicitelly $MC(f)$ is the complex  
$$ \qquad 0 \longrightarrow K^{-1}~ \stackrel{(f^{-1},-d^K)}{\longrightarrow}
~ L^{-1} + K^0 ~ \stackrel{(d^L,f^0)}{\longrightarrow}
~ L^0 \longrightarrow 0$$
concentrated in degree -2,-1 and 0. \\
Let $U$ be an object of $\bS$. The objects of $pst(\tau_{\geq -1}MC(f))(U)$ are the elements of $L^0(U)$ and so they are the same objects of $pst(L)(U)$. Moreover, if $l$ and $l'$ are two objects of
$pst(\tau_{\geq -1}MC(f))(U)$, an arrow of $pst(\tau_{\geq -1}MC(f))(U)$ from $l$ to $l'$ is an isomorphism class of pairs $(\alpha,k)$ with
$k$ an object of $K^0(U)$ and $\alpha$ an object of $L^{-1}(U)$ such that 
\[ (d^L,f^0)(\alpha,k) =l'-l
\] 
This equality can be rewritten as $d^L(\alpha) =l'-(l+f^0(k))$. Therefore an arrow 
from $l$ to $l'$ is an isomorphism class of pairs $(\alpha,k)$ with
$k$ an object of $pst(K)(U)$ and $\alpha: l +F(k) \rightarrow l'$ an arrow of $pst(L)(U)$. According to Definition \ref{def:kercoker}, we can conclude that  
  $pst(\tau_{\geq -1}MC(f)) \cong \coker'(F)$.\\
The objects of $pst(\tau_{ \leq 0}\big( MC(f)[-1]\big))(U)$ are pairs $(f,k)$ with $k$ an object of $pst(K)(U)$ and 
$f:F(k) \rightarrow e_{pst(L)}$ an isomorphism from $F(k)$ to the neutral object $e_{pst(L)}$ of $pst(L)$. If $(f,k)$ and $(f',k')$ are two objects of $pst(\tau_{ \leq 0}\big( MC(f)[-1]\big))(U) $, an arrow of $pst(\tau_{ \leq 0}\big( MC(f)[-1]\big))(U) $ from $(f,k)$ to $(f',k')$ is an element $ g$ of $K^{-1}(U)$ such that 
\[ (f^{-1},-d^K)(g) =(f',k')-(f,k).
\] 
This equality implies the equalities $f^{-1}(g)=f'-f$ and $-d^K(g) =k'-k$. Therefore 
$g: k' \rightarrow k$ is an arrow of $pst(K)(U)$ such that the following diagram is commutative:
\[    \xymatrix{
     F(k')   \ar[rd]_{f} \ar[rr]^{F(g)}&  & F(k) \ar[ld]^{f'}  \\
  &  e_{pst(L)}  & 
}\] 
According to Definition \ref{def:kercoker}, we can conclude that  
  $pst(\tau_{ \leq 0}\big( MC(f)[-1]\big)) \cong \ker(F)$.\\
\end{proof}

By the above Lemma, the following notion of extension in $\cK^{[-1,0]}(\bS)$ is equivalent to Definition \ref{def:extpic} through the equivalence of categories (\ref{st}):

\begin{defn}\label{def:extcom}
Let
\[ K \stackrel{i}{\longrightarrow} L \stackrel{j}{\longrightarrow} M \]
be morphisms of complexes of $\cK^{[-1,0]}(\bS)$. The complex $L$ is an \emph{extension} of $M$ by $K$ if $j \circ i=0$ and 
the following equivalent conditions are satisfied: 
\begin{description}
	\item[(a)] ${\h}^{0}(j): {\h}^{0}(L) \rightarrow {\h}^{0}(M)$ is surjective and $i$ induces a quasi-iso\-mor\-phism between $K$ and $ \tau_{\leq 0} (MC(j)[-1])$;
      \item[(b)] ${\h}^{-1}(i): {\h}^{-1}(K) \rightarrow {\h}^{-1}(L)$ is injective and $j$ induces a quasi-iso\-mor\-phism between $ \tau_{\geq -1} MC(i)$ and $M$.
\end{description}
\end{defn}

\begin{rem}\label{rem:exact-ext} Consider a short exact sequence of complexes in $\cK^{[-1,0]}(\bS)$
\[0 \longrightarrow  K \stackrel{i}{\longrightarrow} L \stackrel{j}{\longrightarrow} M \longrightarrow 0.\]
It exists a distinguished triangle $K  \stackrel{i}{\rightarrow}L \stackrel{j}{\rightarrow} M \rightarrow +$ in $\cD(\bS)$, and $M$ is isomorphic to $MC(i)$ in $\cD(\bS)$. Therefore a short exact sequence of complexes 
in $\cK^{[-1,0]}(\bS)$ is an extension of complexes of $\cK^{[-1,0]}(\bS)$ according to the above definition.  
\end{rem}

\begin{rem} If $K=[K^{-1} \stackrel{0}{\rightarrow} K^0]$ and $M=[M^{-1} \stackrel{0}{\rightarrow} M^0]$, then an extension of $M$ by $K$ consists of an extension of $M^0$ by $K^0$ and an extension of $M^{-1}$ by $K^{-1}.$
\end{rem}

\begin{rem}
Assume that $st(L)=(st(L),I,J)$ is an extension of $st(M)$ by $st(K)$.
Since ${\h}^{-1}(i) $ is injective and ${\h}^{1}(K)=0$, the distinguished triangle $K  \stackrel{i}{\rightarrow}L \rightarrow MC(i) \rightarrow +$
furnishes the long exact sequence
\[ 0 \longrightarrow {\h}^{-1}(K) \stackrel{{\h}^{-1}(i)}{\longrightarrow} {\h}^{-1}(L) \longrightarrow {\h}^{-1}(\tau_{\geq -1} MC(i)) \longrightarrow \]
\[ \longrightarrow {\h}^{0}(K) \stackrel{{\h}^{0}(i)}{\longrightarrow} {\h}^{0}(L) \longrightarrow {\h}^{0}(\tau_{\geq -1} MC(i)) \longrightarrow 0. 
\]
Because of the equality $\tau_{\geq -1} MC(i) = M$ in $\cD(\bS)$, we see that the above long exact sequence is just the long exact sequence (\ref{prop:longexseq}).
\end{rem}

\section{Operations on extensions of strictly commutative Picard stacks}

Using the results of \S 2 we define the pull-back and the push-down of extensions of strictly commutative Picard $\bS$-stacks via additive functors. 
Let $\eic=(\eic,I: \qic \rightarrow \eic,J: \eic \rightarrow \pic)$ be an extension of $\pic$ by $\qic$.

\begin{defn} The \emph{pull-back $F^*\eic$ of the extension $\eic$ via an additive functor $F:\pic' \rightarrow \pic$} is the fibered product $\eic \times_\pic \pic'$
of $\eic$ and $\pic'$ over $\pic$ via $J$ and $F$.
\end{defn}

\begin{lem} 
The pull-pack $F^*\eic $ of $\eic$ via $F$ is an extension of $\pic'$ by $\qic.$
\end{lem}

\begin{proof} Denote by $Pr: F^*\eic \rightarrow \pic' $ the additive functor underlying the pull-back of $\eic$ via $F$. Composing the equivalence of strictly commutative Picard $\bS$-stacks $\qic \cong \ker(J)= \eic \times_\pic \mathbf{1}$ with the natural equivalence of strictly commutative Picard $\bS$-stacks $ \eic \times_\pic \mathbf{1} \cong \eic \times_\pic  \pic' \times_{\pic'} \mathbf{1} = \ker(Pr)$, we get that $\qic$ is equivalent to the 
strictly commutative Picard $\bS$-stack $\ker(Pr)$. Moreover the surjectivity of
$\pi_0(J): \pi_0(\eic) \rightarrow \pi_0(\pic)$ implies the surjectivity of  $\pi_0(Pr): \pi_0(F^*\eic) \rightarrow \pi_0(\pic') $.
Hence $(F^*\eic,I,Pr) $ is an extension of $\pic'$ by $\qic$. 
\end{proof}

\begin{defn} The \emph{push-down $G_*\eic$ of the extension $\eic$ via an additive functor $G:\qic \rightarrow \qic'$} is the fibered sum $\eic +_\qic \qic'$ of $\eic$ and $\qic'$ under $\qic$ via $G$ and $I$.
\end{defn}

\begin{lem} \label{lem:push-down}
The push-down $G_*\eic $ of $\eic$ via $G$ is an extension of $\pic$ by $\qic'.$
\end{lem}

\begin{proof} Denote by $In: \qic' \rightarrow  G_*\eic$ the additive functor underlying the push-down of $\eic$ via $G$. Composing the equivalence of strictly commutative Picard $\bS$-stacks $\coker(I) \cong \pic= \mathbf{1} +^\qic \eic $ with the natural equivalence of strictly commutative Picard $\bS$-stacks $ \mathbf{1}  +^\qic \eic  \cong \mathbf{1} +^{\qic'} \qic' +^\qic \eic = \coker(In)$, we get that $\pic$ is equivalent to the 
strictly commutative Picard $\bS$-stack $\coker(In)$. Moreover the injectivity of
$\pi_1(I): \pi_0(\qic) \rightarrow \pi_1(\eic)$ implies the injectivity of  $\pi_1(In): \pi_1(\qic') \rightarrow \pi_1(G_*\eic) $.
Hence $(G_*\eic,In,P) $ is an extension of $\qic'$ by $\pic$. 
\end{proof}

Before to define a group law for extensions of $\pic$ by $\qic$, we need the following

\begin{lem}\label{lem:product} Let $\eic$ be an extension of $\pic$ by $\qic$ and let $\eic'$ be an extension of $\pic'$ by $\qic'$. Then 
$\eic \times \eic'$ is an extension of $\pic \times \pic'$ by $\qic \times \qic'.$
\end{lem}

\begin{proof} Via the equivalence of categories (\ref{st}), we have that the complex $[\eic]=([\eic],i,j)$ (resp. $[\eic']=([\eic'],i',j')$) is an extension of $[\pic]$ by $[\qic]$ ( resp. an extension of $[\pic']$ by $[\qic']$) in the derived category $\cD(\bS)$. Therefore ${\h}^0(j+j')= {\h}^0(j) +{\h}^0(j'): {\h}^0([\eic]+[\eic']) \rightarrow {\h}^0([\pic] +[\pic'])$ is surjective. Moreover 
$i+i'$ induces an iso\-mor\-phism in $\cD(\bS)$ between $[\qic]+[\qic']$ and 
$$ \tau_{\leq 0} \big(MC(j)[-1]\big)+\tau_{\leq 0} \big(MC(j')[-1]\big)= \tau_{\leq 0} \big(MC(j+j')[-1]\big).$$ 
Hence we can conclude that $[\eic \times \eic' ]=([\eic \times \eic' ],i+i',j+j') $ is an extension of $[\pic\times \pic']$ by $[\qic \times \qic'].$
\end{proof}

Let $\eic$ and $\eic'$ be two extensions of $\pic$ by $\qic$. According to the above lemma, the product $\eic \times \eic'$ is an extension of the product $\pic \times \pic$ by the product $\qic \times \qic$.

\begin{defn} \label{def:+} 
 The \emph{sum $\eic + \eic'$ of the extensions $\eic$ and $\eic'$} is the following extension of $\pic$ by $\qic$
\begin{equation}
D^* +_* ( \eic \times \eic')
\label{eq:+}
\end{equation}
 where $D: \pic \rightarrow \pic \times \pic $ is the diagonal additive functor and $+: \qic \times_\bS \qic \rightarrow \qic$ is the functor underlying the strictly commutative Picard $\bS$-stack $\qic=(\qic,+,\sigma,\tau)$.
\end{defn}

\begin{lem}\label{lem:sumofext}
The above notion of sum of extensions defines on the set of equivalence classes of extensions of $\pic$ by $\qic$ an associative, commutative group law with neutral object, that we denote $\pic \times \qic$.
\end{lem}

\begin{proof} Neutral object: it is the product $\pic \times \qic$ of the extension $\pic=(\pic, \mathbf{1}:\mathbf{1} \rightarrow \pic, Id:\pic \rightarrow \pic)$ of $\pic$ by $\mathbf{1}$ with the extension 
	$\qic=(\qic, Id:\qic \rightarrow \qic, \mathbf{1}:\qic \rightarrow \mathbf{1})$ of $\mathbf{1}$ by $\qic$. Lemma \ref{lem:product} provides that such a product is an extension of $\pic \times \mathbf{1} \cong \pic$ by $\qic \times \mathbf{1} \cong \qic$. Commutativity: it is clear from the formula (\ref{eq:+}).
Associativity: Consider three extensions $\eic, \eic',\eic''$ of $\pic$ by $\qic$.
Using the functor $+: \qic \times_\bS  \qic \times_\bS \qic \rightarrow \qic$ and the diagonal functor $D: \pic \rightarrow \pic \times \pic \times \pic$, it is enough  to show that the extensions $(\eic+ \eic')+\eic''$ and $\eic + (\eic'+\eic'')$ are equivalent. We left this computation to the reader.
\end{proof}

This last Lemma implies that if $\oic,\pic$ and $\qic$ are three strictly commutative Picard $\bS$-stacks, we have the following equivalence of 2-categories
\[{\cExt}(\oic \times \pic,\qic) \cong {\cExt}(\oic,\qic) \times {\cExt}(\pic,\qic),\]
\[{\cExt}(\oic,\pic \times \qic) \cong {\cExt}(\oic,\pic) \times {\cExt}(\oic,\qic).\]
A 2-groupoid is a 2-category whose 1-arrows are invertible up to a 2-arrow and whose 2-arrows are strictly invertible.
An $\bS$-2-stack in 2-groupoids $\PP$ is a fibered 2-category in 2-groupoids over $\bS$ such that 
\begin{itemize}
	\item for every pair of objects $X,Y$ of the 2-category $\PP(U)$, the fibered category of morphisms $\mathrm{Arr}_{\PP(U)}(X,Y)$ of $\PP(U)$ is a $U$-stack (called the $U$-stack of morphisms);
	\item 2-descent is effective for objects in $\PP$.
\end{itemize}
See \cite{B09} \S 6 for more details.
A strictly commutative Picard $\bS$-2-stack is the 2-analog of a
strictly commutative
Picard $\bS$-stack, i.e. it is an $\bS$-2-stack in 2-groupoids $\PP$ endowed with a morphism of $\bS$-2-stacks $ +: {\PP} \times_{\bS} {\PP} \rightarrow {\PP}$ and with
associative and commutative constraints (see \cite{T} Definition 2.3 for more details).
With these notation Lemma \ref{lem:sumofext} implies that extensions of $\pic$ by $\qic$ form a strictly commutative
Picard $\bS$-2-stack $\underline{{\cExt}}(\pic,\qic)$
where
\begin{itemize}
	\item for any object $U$ of $\bS$, the objects of the 2-category $\underline{{\cExt}}(\pic,\qic)(U)$ are extensions
of $\pic_{|U}$ by $\qic_{|U}$, its 1-arrows are additive
functors between such extensions and its 2-arrows are morphisms of additive
functors. In particular if $\eic$ and $\eic'$ are two objects of $\underline{{\cExt}}(\pic,\qic)(U)$, the $U$-stack of morphisms from $\eic$ to $\eic'$ is the $U$-stack ${\HOM}(\eic,\eic')$;
      \item the functor $+ :\underline{{\cExt}}(\pic,\qic)  \times \underline{{\cExt}}(\pic,\qic) \rightarrow \underline{{\cExt}}(\pic,\qic)$ is defined by the
formula (\ref{eq:+}).
\end{itemize}
As for strictly commutative Picard $\bS$-stacks and complexes of abelian sheaves concentrated in degrees -1 and 0, in \cite{T} Tatar proves that there is a dictionary between strictly commutative Picard $\bS$-2-stacks and complexes of abelian sheaves concentrated in degrees -2, -1 and 0. The complex 
of abelian sheaves associated to the strictly commutative
Picard $\bS$-2-stack $\underline{{\cExt}}(\pic,\qic)$
is $\tau_{\leq 0}{\R}{\Hom}\big([\pic],[\qic][1]\big).$


\section{Proof of the main theorem}

In this section we use the same notation as in the introduction.

\begin{defn}
Two extensions $\eic$ and $\eic'$ of $\pic$ by $\qic$
are \emph{equivalent as extensions of $\pic$ by $\qic$} if there is 
\begin{enumerate}
	\item an additive functor $F:\eic \rightarrow \eic'$ and 
	\item  two isomorphisms of additive functors $J' \circ F \cong Id_\pic \circ J$ and $F \circ I \cong I' \circ Id_\qic$,
\end{enumerate}
which are compatible with the isomorphisms of additive functors $J \circ I \cong 0$ and $J' \circ I' \cong 0$ underlying the extensions $\eic$ and $\eic'$ (see \ref{def:morfismi}).
\end{defn}

The additive functor $F$ furnishes the following commutative diagram modulo isomorphisms of additive functors 
\[\xymatrix{
   \qic  \ar[d]_{Id_\qic} \ar[r]^{I} & \eic \ar[d]_{F} \ar[r]^{J} & \pic  \ar[d]^{Id_\pic}  \\
  \qic \ar[r]_{I'} & \eic' \ar[r]_{J'} & \pic .
}
\]

\begin{defn}
An extension of $\pic$ by $\qic$ is \emph{split} if it is equivalent as extension of $\pic$ by $\qic$ to the neutral object $\pic \times \qic$ of the group law defined in \ref{def:+}. 
\end{defn}

\emph{Proof of Theorem \ref{thm:intro}} \textbf{b} and \textbf{c}. Let $\eic=(I:\qic \rightarrow \eic,J:\eic \rightarrow \pic)$ be an extension of $\pic$ by $\qic$.
The strictly commutative Picard $\bS$-stacks ${\HOM}(\pic,\qic)$ and ${\HOM}(\eic,\eic)$ are equivalent as strictly commutative Picard $\bS$-stacks via the following additive functor 
\begin{eqnarray}
\nonumber {\HOM}(\pic,\qic) &\longrightarrow & {\HOM}(\eic,\eic) \\
 \nonumber  F & \mapsto & \big(\eic \rightarrow \eic + IFJ \eic \big).
\end{eqnarray}
By Lemma \ref{lem:HOM} we can conclude that $[{\HOM}(\eic,\eic)]= \tau_{\leq 0}{\R}{\Hom}\big([\pic],[\qic]\big)$, i.e. the group of isomorphism classes of additive functors from $\eic$ to itself is isomorphic to the group ${\Hom}_{\cD(\bS)}([\pic],[\qic])$, and the 
 group of automorphisms of an additive functor from $\eic$ to itself is isomorphic to the group ${\Hom}_{\cD(\bS)}([\pic],[\qic][-1])$ (for this last isomorphism see in particular (\ref{eq:homotopies})). \\

\emph{Proof of Theorem \ref{thm:intro}} \textbf{a}. First we construct a morphism from the group ${\cExt}^1(\pic,\qic)$ of equivalence classes of extensions of $\pic$ by $\qic$ to the group ${\Ext}^1([\pic],[\qic])$
\[ \Theta: {\cExt}^{1}(\pic,\qic) \longrightarrow {\Ext}^1([\pic],[\qic])
\]
and a morphism from the group ${\Ext}^1([\pic],[\qic])$ to the group ${\cExt}^1(\pic,\qic)$
\[ \Psi:  {\Ext}^1([\pic],[\qic]) \longrightarrow {\cExt}^{1}(\pic,\qic)
\]
 Then we check that $\Theta \circ \Psi = Id = \Psi \circ \Theta $ and that $\Theta$ is an homomorphism of groups.
 
(1) Construction of $\Theta$: Consider an extension $\eic=(I:\qic \rightarrow \eic,J:\eic \rightarrow \pic)$ of $\pic$ by $\qic$ and denote by $L=(i:K\rightarrow L,j:L\rightarrow M)$ the corresponding extension of complexes in $\cD^{[-10]}(\bS)$.  By definition we have the equality $K=\tau_{\leq 0} ( MC(j)[-1])$ in the category $\cD^{[-10]}(\bS)$. Hence the distinguished triangle $MC(j)[-1] \rightarrow L \stackrel{j}{\rightarrow} M \rightarrow +$ furnishes the long exact sequence
\begin{equation} \label{eq:thm1}
\cdots \rightarrow {\Hom}(M,K) \rightarrow  {\Hom}(M,L) \stackrel{j \circ}{\rightarrow}   {\Hom}(M,M) \stackrel{\partial}{\rightarrow}  {\Ext}^1(M,K) \rightarrow \cdots 
\end{equation}
We set 
$$\Theta(\eic)= \partial (id_M).$$
 The naturality of the connecting map $\partial$ implies that $\Theta(\eic)$ depends only on the equivalence class of the extension $\eic$.

\begin{lem} If ${\Ext}^1([\pic],[\qic])=0$, then every extension of $\pic$ by $\qic$ is split.
\end{lem}

\begin{proof} By the long exact sequence (\ref{eq:thm1}),
if the cohomology group ${\Ext}^1(M,K)$ is zero, the identity morphisms $id_M: M \rightarrow M$ lifts to a morphism $f : M \rightarrow L$ of $\cD^{[-10]}(\bS)$ which corresponds via the equivalence of categories (\ref{st}) to an isomorphism classes of additive functors $F: \pic \rightarrow \eic$ such that $J \circ F \cong Id_{\pic}$. Hence $\eic$ is a split extension of $\pic$ by $\qic$.
\end{proof}

The above Lemma means that $\Theta(\eic)$ is an obstruction for the extension $\eic$  to be split: $\eic$ is split if and only if $id_M: M \rightarrow M$ lifts to ${\Hom}(M,L)$ if and only if $\Theta(\eic)$ vanishes in ${\Ext}^1([\pic],[\qic])$.

(2) Construction of $\Psi$: Choose two complexes $P=[ P^{-1} \stackrel{d_P}{\rightarrow} P^0]$ and $N=[N^{-1} \stackrel{d_N}{\rightarrow} N^0]$ of $\cD^{[-10]}(\bS)$ such that
$P^{-1}, P^0$ are projective and the three complexes $N,P,M$ build a short exact sequence in $\cD^{[-10]}(\bS)$
\begin{equation}
0 \longrightarrow N  \stackrel{s}{\longrightarrow} P \stackrel{t}{\longrightarrow} M \longrightarrow 0
\label{eq:thm2}
\end{equation}
(because of the projectivity of the $P^i$ for $i=-1,0$, there exists a surjective morphism of complexes $P \rightarrow M$ and then, in order to get a short exact sequence, choose $N^i = \ker( P^i \rightarrow M^i)$ for $i=-1,0$). 
By Remark \ref{rem:exact-ext} the above exact sequence furnishes an extension of strictly commutative Picard $\bS$-stacks 
\[  st(N)  \stackrel{S}{\longrightarrow} st(P) \stackrel{T}{\longrightarrow} \pic \]
where $S$ and $T$ are the isomorphism classes of additive functors corresponding to the morphisms $s$ and $t$. Applying ${\Hom}(-,K[1])$ to the distinguished triangle 
$N  \rightarrow P \rightarrow M \rightarrow +$ associated to the
short exacts sequence (\ref{eq:thm2}) we get the long exact sequence
\begin{equation}
  \cdots \rightarrow {\Hom}(M,K)  \rightarrow {\Hom}(P,K) \stackrel{ \circ s}{\rightarrow} {\Hom}(N,K) \stackrel{\partial}{\rightarrow} {\Ext}^1(M,K) \rightarrow 0.
\label{eq:thm3}
\end{equation}
Given an element $x$ of ${\Ext}^1(M,K)$, choose an element $u$ of ${\Hom}(N,K)$ such that $\partial (u) =x$. We set
\[ \Psi(x)=  U_*st(P), \]
i.e. $\Psi(x)$ is the push-down $U_*st(P)$ of the extension $st(P)$ via one representative of the isomorphism class $U:st(N) \rightarrow \qic$ of additive functors corresponding to the morphism $u:N \rightarrow K$ of $\cD^{[-10]}(\bS)$.
 By Lemma \ref{lem:push-down} the strictly commutative Picard $\bS$-stack $\Psi(x)$ is an extension of $\pic$ by $\qic$. Now we check that the morphism $\Psi$ is well defined, i.e. $\Psi(x)$ doesn't depend on the lift of $x$.  If $u' \in {\Hom}(N,K)$ is another lift of $x$, then there exists an element $f$ of ${\Hom}(P,K)$
 such that $u'-u=f \circ s.$ Consider the push-down $(U'-U)_*st(P)$ of the extension $st(P)$ via one representative of the isomorphism class $U'-U:st(N) \rightarrow \qic$ of additive functors (as for $u$, we denote here by $U':st(N) \rightarrow \qic$ the isomorphism class corresponding to the morphism $u':N \rightarrow K$ of $\cD^{[-10]}(\bS)$). Since $u'-u=f \circ s$, by the universal property of the push-down there exists a unique additive functor $H: (U'-U)_*st(P) \rightarrow \qic$ such that $H \circ In \cong Id_\qic$, where $In: \qic \rightarrow (U'-U)_*st(P)$ is the additive functor underlying the extension $(U'-U)_*st(P)$ of $\pic$ by $\qic$. Hence 
 the extension $(U'-U)_*st(P)$ of $\pic$ by $\qic$ is split and so the extensions $U'_*st(P)$ and $U_*st(P)$ are equivalent.
 
(3) $ \Theta \circ \Psi = Id$: With the notation of (2), given an element $x$ of ${\Ext}^1(M,K)$, choose an element $u$ of ${\Hom}(N,K)$ such that $\partial (u) =x$. 
By definition $\Psi(x)=  U_*st(P)$. Because of the naturality of the connecting map $\partial$, the following diagram commutes
\[\xymatrix{
  {\Hom}(M,K)  \ar@{=}[d] \ar[r] & {\Hom}(P,K)  \ar[r]& {\Hom}(N,K) \ar[r]^{\partial} &{\Ext}^1(M,K) \ar@{=}[d] \\
  {\Hom}(M,K) \ar[r] & {\Hom}([U_*st(P)],K) \ar[r] \ar[u] & {\Hom}(K,K) \ar[u]^{\circ u}  \ar[r]^{\partial} & {\Ext}^1(M,K)
}
\]
Therefore $\Theta(U_*st(P))=x$, i.e. $\Theta$ surjective.

(4) $\Psi \circ \Theta= Id$: Consider an extension $\eic=(I:\qic \rightarrow \eic,J:\eic \rightarrow \pic)$ of $\pic$ by $\qic$ and 
Denote by $L=(i:K\rightarrow L,j:L\rightarrow M)$ the corresponding extension of complexes in $\cD^{[-10]}(\bS)$. Choose two complexes $P=[ P^{-1} \rightarrow P^0]$ and $N=[N^{-1} \rightarrow N^0]$ as in (2). The lifting property of the complex $P$ furnishes a lift $u:P \rightarrow L$ of the morphism of complexes $t:P \rightarrow M$ and hence a commutative diagram 
\[\xymatrix{
   st(N)  \ar[d]_{U_|} \ar[r]^{S} & st(P) \ar[d]_{U} \ar[r]^{T} & \pic  \ar[d]^{Id_\pic}  \\
  \qic \ar[r]_{I} & \eic \ar[r]_{J} & \pic .
}
\]
where $U:st(P) \rightarrow \eic$ is the isomorphism class of additive functors corresponding to the lift $u:P \rightarrow L$ and $U_|:st(N) \rightarrow \qic$
is the restriction of $U$ to $st(N)$. Consider now the push-down $(U_|)_*st(P)$ of 
the extension $st(P)$ via a representative of $U_|$. Because of the universal property of the push-down, there exists a unique additive functor $H: (U_|)_* st(P) \rightarrow \eic$ such that the following diagram commute
\[\xymatrix{
   st(N)  \ar[dd]^{U_|} \ar[dr]^{U_|} \ar[rr]^{S}&  & st(P) \ar@/_/[dd]|-{U} \ar[dr] \ar[rr]^{T} & & \pic  \ar@/_/[dd]|-{Id_\pic} \ar[dr]^{Id_\pic}& \\
 & \qic  \ar[rr] \ar[dl]^{Id_\qic} &  & \qic +^{st(N)} st(P) \ar@{.>}[dl]^{H} \ar[rr] & & \pic  \ar[dl]^{Id_\pic}  \\
  \qic \ar[rr]_{I} & & \eic \ar[rr]_{J} & & \pic .& 
}
\]
Hence we have that the extensions $(U_|)_*st(P)$ and $\eic$ are equivalent, which implies that $\Psi (\Theta(\eic))=\Psi (\Theta((U_|)_*st(P)))=(U_|)_*st(P) \cong \eic$, i.e. $\Theta$ injective.

(5) $\Theta$ is an homomorphism of groups: Consider two extensions $\eic, \eic'$ of $\pic$ by $\qic$. With the notations of (2) we can suppose that $\eic= U_*st(P)$ and 
$\eic'= U'_*st(P)$ with $U,U':st(N) \rightarrow \qic$ two isomorphism classes of additive functors corresponding to two morphisms $u,u':N \rightarrow K$ of $\cD^{[-10]}(\bS)$. Now by Definition \ref{def:+}
\begin{eqnarray}
\nonumber \eic + \eic' & = & D^* (+_\qic)_*  \big(U_*st(P) \times U'_*st(P)\big)\\
\nonumber  & = & D^* (+_\qic)_* (U \times U')_* (st(P) \times st(P))\\
\nonumber  & = & (U + U')_* D^* (+_{st(N)})_*  (st(P) \times st(P))\\
\nonumber  & = & (U + U')_*  (st(P) + st(P))
\end{eqnarray}
where $D: \pic \rightarrow \pic \times \pic $ is the diagonal additive functor and $+_\qic: \qic \times_\bS \qic \rightarrow \qic$ (resp. $+_{st(N)}: st(N) \times_\bS st(N) \rightarrow st(N)$) is the functor underlying the strictly commutative Picard $\bS$-stack $\qic$ (resp. $st(N)$). If $\partial:{\Hom}(N,K) \rightarrow {\Ext}^1(M,K)$ is the connecting map of the long exact sequence (\ref{eq:thm3}), we get 
\[ \Theta ( \eic + \eic') = \partial (u+u') = \partial (u) +\partial(u')= \Theta ( \eic )+ \Theta( \eic') .\]   

\begin{rem} In the construction of $\Psi: {\Ext}^1([\pic],[\qic]) \rightarrow {\cExt}^{1}(\pic,\qic)$, instead of two complexes $P=[ P^{-1} \rightarrow P^0]$ and $N$ of $\cD^{[-10]}(\bS)$ such that
$P^{-1}, P^0$ are projective and the three complexes $N,P,M$ build a short exact sequence $ 0 \rightarrow N \rightarrow P \rightarrow M \rightarrow 0$, we can consider two complexes $I=[ I^{-1} \rightarrow I^0]$ and $N'$ of $\cD^{[-10]}(\bS)$ such that
$I^{-1}, I^0$ are injective and the three complexes $K,I,N'$ build a short exact sequence $ 0 \rightarrow K \rightarrow I \rightarrow N' \rightarrow 0.$  In this case instead of applying ${\Hom}(-,K[1])$ we apply  ${\Hom}(M,-)$ and instead of considering push-downs of extensions we consider pull-backs. This two way to construct $\Psi$ with projectives or with injectives are dual.
\end{rem} 


\end{document}